\documentclass[11pt, reqno]{amsart}

\usepackage{amsmath}
\usepackage{amssymb}
\usepackage{amsthm}
\usepackage{mathrsfs}
\usepackage{mathtools}
\usepackage{bm}
\usepackage{bbm}
\usepackage{dsfont}
\usepackage{graphicx}
\usepackage{enumitem}
\usepackage{comment}
\usepackage{mleftright}
\usepackage{color}
\usepackage{iftex}

\ifuptex
	\usepackage[dvipdfmx, hypertexnames=false, hidelinks]{hyperref}
	\hypersetup{colorlinks=true,citecolor=blue,linkcolor=blue,urlcolor=blue,
		pdfstartview=FitH }
\else
\ifpdftex
\usepackage
	{hyperref}
	\hypersetup{colorlinks=true,citecolor=blue,linkcolor=blue,urlcolor=blue,
		pdfstartview=FitH }
\fi\fi

\usepackage{cleveref}

\usepackage{autonum}

\allowdisplaybreaks[3]
\theoremstyle{plain} 
\newtheorem{theorem}{Theorem}
\newtheorem{lemma}{Lemma}[section]

\newtheorem{proposition}{Proposition}

\newtheorem*{conjecture*}{Conjecture}
\newtheorem*{theorem*}{Theorem}
\newtheorem*{question*}{Question}

\theoremstyle{plain}

\theoremstyle{remark}
\newtheorem{remark}{Remark}

\theoremstyle{definition}
\newtheorem*{assumption*}{Assumption}

\newtheorem*{notations*}{Notations}
\newtheorem*{acknowledgment*}{Acknowledgments}

\numberwithin{equation}{section}

\crefformat{enumi}{\textnormal{(\textcolor{blue}{#1})}}

\crefname{section}{Section}{Sections}

\crefname{theorem}{Theorem}{Theorems}
\crefname{corollary}{Corollary}{Corollaries}
\crefname{lemma}{Lemma}{Lemmas}
\crefname{proposition}{Proposition}{Propositions}
\crefname{claim}{Claim}{Claims}
\crefname{definition}{Definition}{Definitions}
\crefname{notation}{Notation}{Notations}
\crefname{problem}{Problem}{Problems}
\crefname{note}{Note}{Notes}
\crefname{remark}{Remark}{Remarks}
\crefname{example}{Example}{Examples}
\crefname{equation}{}{}

\crefname{enumi}{}{}
\crefname{enumii}{}{}
\crefname{enumiii}{}{}

\newcommand\swapcommand[2]{%
\let\swaptemp#1
\let#1#2
\let#2\swaptemp
}

\let\sl\l
\DeclareRobustCommand{\l}{%
	\ifmmode
		\mleft
	\else
		\sl
	\fi
}

\let\sL\L
\DeclareRobustCommand{\L}{%
	\ifmmode
	\mathscr{L}
	\else
		\sL
	\fi
}

\let\rmi\i
\DeclareRobustCommand{\i}{%
	\ifmmode
	\mathrm{i}
	\else
		\rmi
	\fi
}

\NewDocumentCommand{\qquadn}{m}{\qquad\ifnum#1>1\qquadn{\numexpr#1-1\relax}\fi}

\NewDocumentCommand{\quadn}{m}{\quad\ifnum#1>1\quadn{\numexpr#1-1\relax}\fi}

\swapcommand{\SS}{\S}
\renewcommand{\S}{\mathscr{S}}

\newcommand\set[2]{%
	\left\{ #1:  #2 \right\}
}

\renewcommand{\vector}[1]{\begin{pmatirx} #1 \end{pmatrix}}

\newcommand{\CC}{\mathbb{C}}
\newcommand{\RR}{\mathbb{R}}

\newcommand{\ZZ}{\mathbb{Z}}

\newcommand{\s}{\sigma}

\newcommand{\lam}{\lambda}
\newcommand{\Lam}{\Lambda}

\newcommand{\ceq}{\coloneqq}

\newcommand{\sd}{\, \mathrm{d}} 

\newcommand{\ds}{\displaystyle}

\renewcommand{\a}{\alpha}
\renewcommand{\b}{\beta}

\renewcommand{\r}{\mright}

\renewcommand{\Re}{\operatorname{Re}}
\renewcommand{\Im}{\operatorname{Im}}
\renewcommand{\epsilon}{\varepsilon}

\renewcommand{\hat}{\widehat}
\renewcommand{\bar}{\overline}
\renewcommand{\sL}{\mathscr{L}}
\let\psetminus\setminus
\renewcommand{\setminus}{\mathbin{\psetminus}}

\title{Small gaps between consecutive zeros of the Riemann zeta-function}
\author[S. Inoue]{Sh\={o}ta Inoue}

\address[S. Inoue]{Department of Liberal Arts and Basic Sciences, College of Industrial Technology, Nihon University, 2-11-1 Shin-ei, Narashino, Chiba 275-8576, Japan}
\email{inoue.shota@nihon-u.ac.jp}

\keywords{Riemann zeta-function, gaps between zeros, resonance method, pair correlation of zeros}

\subjclass{Primary 11M26; Secondary 11M06}

\begin{document}

\begin{abstract}
	In this paper, we introduce the \textit{resonance-correlation method} to study small gaps between consecutive zeros of the Riemann zeta-function.
	Our method is based on a synthesis of Montgomery's pair correlation approach and the Montgomery--Odlyzko method.
	As an application, we break the persistent practical barrier around $0.515$ and prove $\mu < 0.50895$ under the Riemann Hypothesis.
\end{abstract}

\maketitle

\section{\textbf{Introduction}}
	The distribution of gaps between consecutive zeros of the Riemann zeta-function $\zeta$ is a central subject in analytic number theory.
	Let $0 < \gamma_{1} \leq \gamma_{2} \leq \cdots \leq \gamma_{n} \leq \cdots$ denote the sequence of ordinates of the nontrivial zeros of $\zeta$ in the upper half-plane.
	By the Riemann--von Mangoldt formula, the average spacing between consecutive zeros $\gamma_{n}, \gamma_{n + 1}$ is asymptotic to $2\pi / \log \gamma_{n}$.
	The extreme variations of the gaps between zeros are captured by the quantities
	\begin{align}
		\lam = \limsup_{n \rightarrow + \infty}\frac{\gamma_{n + 1} - \gamma_{n}}{2\pi / \log\gamma_{n}},
		\quad
		\mu = \liminf_{n \rightarrow + \infty}\frac{\gamma_{n + 1} - \gamma_{n}}{2\pi / \log\gamma_{n}}.
	\end{align}
	The trivial bound is $\mu \leq 1 \leq \lam$.

	Assuming the Riemann Hypothesis (RH), Montgomery \cite{Mo1973} established the framework of the pair correlation of zeros and showed $\mu \leq 0.68$.
	In the same paper, he also proposed a conjecture called the Pair Correlation Conjecture (PCC).
	Based on a connection with random matrix theory observed jointly with Dyson, Montgomery further conjectured that $\lam = \infty$ and $\mu = 0$.
	Toward understanding $\mu$ and $\lam$, Montgomery and Odlyzko \cite{MO1984} established a method, which led to $\mu < 0.5179$ and $\lam > 1.9799$ under RH.
	For the large gaps $\lam$, significant breakthroughs have since been achieved under RH, both through the development of alternative approaches (e.g., Hall \cite{Ha2005}) and by eventually breaking the barrier of the Montgomery--Odlyzko method (e.g., Bui and Milinovich \cite{BM2017}, see also \cite{IKT2025}).

	In contrast, progress on the small gaps $\mu$ has been gradual over the past four decades.
	While the Montgomery--Odlyzko method has served as the foundation for studying $\mu$, refinements in the choice of weight functions by various authors have gradually lowered the bound.
	Notably, under RH, the record progressed from $\mu < 0.5172$ by Conrey, Ghosh, and Gonek \cite{CGG1984} to $\mu < 0.5155$ by Bui, Milinovich, and Ng \cite{BMN2010},
	and $\mu < 0.515398$ by Feng and Wu \cite{FW2012}, culminating in the current best conditional result $\mu < 0.515396$ by Preobrazhenski\u{\i} \cite{Pr2016}.
	These studies suggest that optimizing weight functions within the Montgomery--Odlyzko method has essentially reached a practical barrier around $0.515$.
	Indeed, while the absolute theoretical limit of the Montgomery--Odlyzko method is rigorously established as $0.508$ in \cite{IKT2025},
	the historical rate of improvement suggests that further refinements of weight functions alone are unlikely to surpass $0.515$ significantly.

	In this paper, we introduce an approach that we call the \textit{resonance-correlation method}.
	For a short interval length $h > 0$, let $N_{h}(t)$ denote the number of zeros of the Riemann zeta-function in the interval $[t - h / 2, t + h / 2]$, which will be defined rigorously in the next section.
	While the Montgomery--Odlyzko method relies on evaluating the linear quantity $N_{h}(t) - 1$,
	our method synthesizes this framework with Montgomery's pair correlation by shifting the focus to the quadratic quantity $N_{h}(t)^{2} - N_{h}(t)$.
	To evaluate this nonlinear quantity effectively, we introduce a tool that we call an \textit{approximator}.
	By constructing a suitable combination of a resonator and this approximator,
	we circumvent the limitations of previous approaches.
	As an application, we break the persistent practical barrier around $0.515$ and obtain the following theorem.

	\begin{theorem}	\label{SGZuRH}
		Under the Riemann Hypothesis, we have $\mu < 0.50895$.
	\end{theorem}

	Finally, we remark on the importance of the $0.5$ barrier, which is deeply related to the Alternative Hypothesis and the Siegel zero problem.
	This connection was initially suggested by Heath-Brown (unpublished work) and Conrey \cite{Co2003}, and was rigorously proved by Conrey and Iwaniec \cite{CI2002}.
	By breaking the $0.515$ bottleneck, this paper has implications for these well-known open problems.
	It remains an interesting open question whether this method, possibly combined with further assumptions such as the Generalized Riemann Hypothesis,
	can strictly break the structural barrier $0.508$ of the Montgomery--Odlyzko method.

\section{\textbf{Overview of the method}}	\label{Sec:APSG}
	Our method is built upon two key ingredients.
	We describe these while reviewing previous works.
	For $T \in \RR$, let $N(T)$ denote the number of zeros $\rho = \b + \i\gamma$ satisfying $0 < \gamma < |T|$ counted with multiplicity.
	If $T$ is equal to the imaginary part of a zero, then $N(T) = (N(T^{+}) + N(T^{-})) / 2$.
	At this point, we also fix the branch of the logarithm of the Riemann zeta-function.
	Let $\log\zeta(\s + \i t) = \int_{\infty}^{\s}\frac{\zeta'}{\zeta}(\a + \i t)\sd\a$ if $t$ is neither equal to $0$ nor the imaginary part of a zero of $\zeta$.
	If $t$ is equal to $0$ or the imaginary part of a zero, then $\log\zeta(\s + \i t) = (\log\zeta(\s + \i t^{+}) + \log\zeta(\s + \i t^{-})) / 2$.
	Let $N_{h}(t) = N(t + \tfrac{h}{2}) - N(t - \tfrac{h}{2})$ for $t \in \RR$ and $h > 0$, which is the principal object of this paper.
	The fundamental idea to estimate large or small gaps between zeros is to find a $t$ such that $N_{h}(t) < 1$ or $N_{h}(t) > 1$.
	In the estimation of $\mu$, Montgomery's pair correlation method \cite{Mo1973} can be essentially regarded as a method that compares the two integrals
	\begin{align}
		\int_{0}^{T}N_{h}(t) \sd t, \quad \int_{0}^{T}N_{h}(t)^{2}\sd t.
	\end{align}
	The latter integral is connected to the pair correlation through the works of Goldston \cite{Go1987} and Fujii \cite{Fu1990}.
	In fact, we can obtain $\mu \leq 0.68$ by combining Montgomery's theorem and Fujii's theorem, which coincides with Montgomery's original bound.
	While the bound for $\mu$ has been lowered to near $0.6$ under RH through the optimization of test functions within the pure pair correlation framework (see, e.g., \cite{CCLM2017, CGL2020, GGOS2000}),
	achieving an essential refinement, such as $\mu < 0.55$, remains difficult due to fundamental obstacles involving PCC.

	On the other hand, Montgomery and Odlyzko \cite{MO1984} considered the two integrals
	\begin{align}
		\int_{-\infty}^{\infty} N_{h}(t) |R_{L}(\tfrac{1}{2} + \i t)|^{2} \Phi_{T}(t)\sd t, \quad \int_{-\infty}^{\infty} |R_{L}(\tfrac{1}{2} + \i t)|^{2} \Phi_{T}(t)\sd t
	\end{align}
	to find a $t$ such that $N_{h}(t) < 1$ or $N_{h}(t) > 1$.
	Here, $\Phi_{T}$ is a smooth nonnegative-valued function, and $R_{L}$ is the resonator defined by a Dirichlet polynomial $R_{L}(s) = \sum_{n \leq L}r(n) n^{-s}$.
	This idea of utilizing a Dirichlet polynomial serves as a vital precursor to the resonance method.
	Soundararajan \cite{So2008} pioneered the resonance method in his seminal work, and Bondarenko and Seip \cite{BS2017} further advanced it.
	However, in the search for small gaps, evaluating the linear quantity $N_{h}(t) - 1$ within the Montgomery--Odlyzko method suffers from a significant loss of efficiency.
	For small $h$, $N_{h}(t)$ is zero for many $t$, which penalizes the integral by yielding negative contributions.

	Our first ingredient synthesizes the pair correlation approach with the Montgomery--Odlyzko method to eliminate such negative contributions by comparing the following two integrals:
	\begin{align}
		M_{1}(T, h; R_{L})
		= \int_{-\infty}^{\infty} N_{h}(t) |R_{L}(\tfrac{1}{2} + \i t)|^{2} \Phi_{T}(t) \sd t,
	\end{align}
	and
	\begin{align}
		M_{2}(T, h; R_{L})
		= \int_{-\infty}^{\infty} N_{h}(t)^{2} |R_{L}(\tfrac{1}{2} + \i t)|^{2} \Phi_{T}(t) \sd t.
	\end{align}
	For this reason, we call this method the resonance-correlation method.
	Then, our goal is to show that $M_{2}(T, h; R_{L}) > M_{1}(T, h; R_{L})$ for small $h$, which guarantees the existence of some $t$ satisfying $N_{h}(t) > 1$.

	In calculating $M_{2}$, a difficulty arises when evaluating the integral
	\begin{align}
		\int_{-\infty}^{\infty}| \log\zeta(\tfrac{1}{2} + \i(t + \tfrac{h}{2})) - \log\zeta(\tfrac{1}{2} + \i(t - \tfrac{h}{2})) |^{2}
		|R_{L}(\tfrac{1}{2} + \i t)|^{2} \Phi_{T}(t) \sd t,
	\end{align}
	which originates from the Riemann--von Mangoldt formula.
	Since obtaining an exact asymptotic formula for this oscillating non-linear term is fundamentally intractable,
	we establish a rigorous and effective lower bound instead.
	To achieve this, we introduce an approximator defined by the Dirichlet polynomial
	\begin{align}	\label{def_AX}
		A_{X}(s)
		= \sum_{n \leq X}\frac{a(n)}{n^{s}}.
	\end{align}
	Here, $a$ is a suitable arithmetic function.
	Using this, we establish the inequality
	\begin{align}
		&\int_{-\infty}^{\infty}| \log\zeta(\tfrac{1}{2} + \i(t + \tfrac{h}{2})) - \log\zeta(\tfrac{1}{2} + \i(t - \tfrac{h}{2})) |^{2}
		|R_{L}(\tfrac{1}{2} + \i t)|^{2} \Phi_{T}(t) \sd t\\
		&= \int_{-\infty}^{\infty} \bigg|\sum_{n \leq L} \l\{ \log\zeta(\tfrac{1}{2} + \i (t + \tfrac{h}{2})) - \log\zeta(\tfrac{1}{2} + \i (t - \tfrac{h}{2})) - A_{L/n}(\tfrac{1}{2} + \i t) \r\}
		r(n) n^{-\i t} \bigg|^{2} \Phi_{T}(t) \sd t\\
		&\quad- \mathcal{I}_{1}(T; R_{L}, A) + 2\Re \mathcal{I}_{2}(T, h; R_{L}, A)\\
		\label{KINEDI}
		&\geq - \mathcal{I}_{1}(T; R_{L}, A) + 2\Re \mathcal{I}_{2}(T, h; R_{L}, A),
	\end{align}
	where
	\begin{align}	\label{def_calI1}
		\mathcal{I}_{1}(T; R_{L}, A)
		= \int_{-\infty}^{\infty} \bigg|\sum_{n \leq L} A_{L / n}(\tfrac{1}{2} + \i t) r(n) n^{-1/2 - \i t} \bigg|^{2} \Phi_{T}(t)\sd t,
	\end{align}
	and
	\begin{align}	\label{def_calI2}
		\mathcal{I}_{2}(T, h; R_{L}, A)
		&= \sum_{m, n \leq L}\frac{r(m) \bar{r(n)}}{\sqrt{mn}} \int_{-\infty}^{\infty} \l\{ \log\zeta(\tfrac{1}{2} + \i(t + \tfrac{h}{2})) - \log\zeta(\tfrac{1}{2} + \i(t - \tfrac{h}{2})) \r\}\\
		&\qquadn{8} \times\bar{A_{L / n}(\tfrac{1}{2} + \i t)} \l( \frac{m}{n} \r)^{-\i t} \Phi_{T}(t)\sd t.
	\end{align}
	This inequality using the approximator is the second ingredient of this method.
	The terminology of the approximator reflects its role as a mean-square approximation factor.
    By employing the elementary inequality $|z|^{2} \geq 2\Re(z\bar{w}) - |w|^{2}$, this approximator allows us to bound the intractable quadratic quantity $N_{h}(t)^{2}$ from below.
    This process yields quantities that can be evaluated analytically.
    While the inequality itself is elementary, a similar strategic use of this technique proved highly effective in bounding moments of the Riemann zeta-function, as demonstrated by Conrey and Ghosh \cite{CG1984}.
    Our novelty lies in combining this principle with the resonance method to establish the resonance-correlation method.

	To apply this method, the coefficients of the resonator $r(n)$ are chosen similarly to those in Bui, Milinovich, and Ng \cite{BMN2010}, which builds on the works of Montgomery--Odlyzko \cite{MO1984}
	and Conrey--Ghosh--Gonek \cite{CGG1984}.
	Moreover, we choose the coefficients of the approximator $a(n)$ as the coefficients of the Dirichlet series of $\log\zeta(s + \i \tfrac{h}{2}) - \log\zeta(s - \i \tfrac{h}{2})$,
	which is the most natural and simplest choice.

	As a result, we obtain the following theorem.

	\begin{theorem}	\label{Nh2-Nh}
		Assume the Riemann Hypothesis.
		Let $\phi > 0$, $\ell \geq 1$,
		and let $f$ be a continuous real-valued function of bounded variation on $[0, 1]$ satisfying $I_{f}(\ell) \ceq \int_{0}^{1}f(u)^{2} u^{\ell^{2} - 1}\sd u \neq 0$.
		For any sufficiently large $T$ and $h = 2\pi \phi / \log T$, we have
		\begin{align}
			\sup_{t \in [T, 2T]}(N_{h}(t)^{2} - N_{h}(t))
			\geq I_{f}(\ell)^{-1}\mathcal{M}_{\ell, f}(\phi) - \phi(1 - \phi) - O_{\ell, f}\l( \frac{\phi^{2}}{\log T} + \frac{\phi \log\log T}{\log T} \r),
		\end{align}
		where
		\begin{align}	\label{def_calM}
			\mathcal{M}_{\ell, f}(\phi)
			&= |1 - 2\phi| \frac{2}{\pi} \ell \int_{0}^{1}\frac{\sin(\pi \phi u)}{u} \int_{0}^{1 - u} f(v) f(u + v) v^{\ell^{2} - 1}\sd v \sd u\\
			&\quad+ \frac{2}{\pi^{2}} \ell^{2} \int_{0}^{1}\frac{\sin(\pi \phi u_{1})}{u_{1}} \int_{0}^{1 - u_{1}}\frac{\sin(\pi \phi u_{2})}{u_{2}}\\
			&\qquadn{1} \times \int_{0}^{1 - (u_{1} + u_{2})} \l\{f(v) f(u_{1} + u_{2} + v) + f(u_{1} + v) f(u_{2} + v)\r\} v^{\ell^{2} - 1}\sd v \sd u_{2} \sd u_{1}\\
			&\quad+\frac{2}{\pi^{2}} \int_{0}^{1} \frac{\sin^{2}(\pi \phi u)}{u}\int_{0}^{1 - u} f(v)^{2} v^{\ell^{2} - 1}\sd v \sd u.
		\end{align}
	\end{theorem}

	At the end of this section, we prove \cref{SGZuRH} using \cref{Nh2-Nh}.

	\begin{proof}[Proof of \cref{SGZuRH}]
		Applying \cref{Nh2-Nh} with $\phi = 0.508949$, $\ell = 1.15$, and $f(x) = 1 + c_{1}x$ for $c_{1} = -0.7$, we have
		\begin{align}
			\sup_{t \in [T, 2T]}(N_{h}(t)^{2} - N_{h}(t))
			\geq I_{f}(\ell)^{-1}\mathcal{M}_{\ell, f}(\phi) - \phi(1 - \phi) - o(1)
			\geq 10^{-5} > 0
		\end{align}
		for any large $T$.
		Hence, for every large $T$, there exist some $\gamma_{n}, \gamma_{n + 1} \in [T - h / 2, 2T + h / 2]$
		such that $\gamma_{n + 1} - \gamma_{n} \leq h = 2\pi \phi / \log T = 2 \pi (\phi + o(1)) / \log\gamma_{n}$.
		This concludes that $\mu \leq \phi = 0.508949 < 0.50895$, which completes the proof of \cref{SGZuRH}.
	\end{proof}

	\begin{remark}
		The constant $0.50895$ is not optimal in our method.
		In fact, while we choose $f$ to be a linear function, we can slightly improve the constant by choosing a higher-degree polynomial.
		However, the primary focus of this paper is to introduce the resonance-correlation method, so we forgo further optimization for the sake of simplicity.
	\end{remark}

\section{\textbf{Asymptotic and lower bounds for $M_{1}(T, h; R_{L})$ and $M_{2}(T, h; R_{L})$}}
	Throughout this paper, the smooth weight function $\Phi_{T}$ is defined by $\Phi_{T}(t) = \Phi\l(\frac{t - 3T / 2}{T / \log T}\r)$ with $\Phi(t) = e^{-t^{2} / 2}$.
	We will freely use the formula $\hat{\Phi_{T}}(0) = \int_{\RR}\Phi_{T}(t) \sd t = \sqrt{2\pi} T / \log T$.
	Moreover, let $g_{h}(k)$ denote the $k$-th Dirichlet coefficient of $\log\zeta(s + \i h / 2) - \log\zeta(s - \i h / 2)$,
	that is, $g_{h}(1) = 0$ and $g_{h}(k) = -2\i \Lam(k) \sin\l(\frac{h}{2}\log k\r) / \log k$ for $k \in \ZZ_{\geq 2}$.
	Here, $\Lam$ is the von Mangoldt function.
	With this setting, we prove an asymptotic formula for $M_{1}(T, h; R_{L})$ and establish a lower bound of $M_{2}(T, h; R_{L})$ in terms of resonators and approximators.

	\begin{theorem}	\label{PAFM1}
		Assume the Riemann Hypothesis.
		Let $r$ be an arithmetic function.
		For any sufficiently large $L, T$ satisfying $L \leq T / (\log T)^{2}$, and any $0 < h \leq 1$, we have
		\begin{align}	\label{AFM1}
			M_{1}(T, h; R_{L})
			&= \hat{\Phi_{T}}(0) \l(\frac{h}{2\pi} \log T \sum_{n \leq L}\frac{|r(n)|^{2}}{n}
			+ \frac{1}{\pi}\Im\sum_{km \leq L}\frac{g_{h}(k)}{k}\frac{r(m) \bar{r(km)}}{m}\r)\\
			&\quad+ O\l( \hat{\Phi_{T}}(0)\l( h + \frac{1}{T} \r)\sum_{n \leq L}\frac{|r(n)|^{2}}{n} \r).
		\end{align}
	\end{theorem}

	\begin{theorem}	\label{PAFM2}
		Assume the Riemann Hypothesis.
		Let $a, r$ be arithmetic functions.
		For any large $L, T$ satisfying $L \leq T / (\log T)^{2}$, and any $0 < h \leq 1$, we have
		\begin{align}	\label{AFM2}
			M_{2}(T, h; R_{L})
			&\geq \l(\frac{h}{\pi}\log T + O(h)\r) M_{1}(T, h; R_{L}) - \hat{\Phi_{T}}(0) \l( \frac{h}{2\pi}\log T \r)^{2} \sum_{n \leq L}\frac{|r(n)|^{2}}{n}\\
			&\quad+ \hat{\Phi_{T}}(0)\frac{1}{2\pi^{2}}\Re\sum_{\substack{km, ln \leq L\\ km = ln}} \frac{(2g_{h}(k) - a(k)) \bar{a(l)}}{\sqrt{kl}} \frac{r(m) \bar{r(n)}}{\sqrt{mn}}\\
			&\quad- \hat{\Phi_{T}}(0) \frac{1}{2\pi^{2}}\Re\sum_{klm \leq L}\frac{g_{h}(k) g_{h}(l)}{kl} \frac{r(m) \bar{r(klm)}}{m}\\
			&\quad- O\l( \hat{\Phi_{T}}(0)\l( h^{2} \log T + \frac{1}{T}\l( \log T + \sum_{k \leq L}\frac{|a(k)|^{2}}{k} \r) \r) \sum_{n \leq L}\frac{|r(n)|^{2}}{n} \r).
		\end{align}
	\end{theorem}

	Although \cref{PAFM1} has essentially been proved in \cite{IKT2025}, we give a proof in the next section for completeness.

\section{\textbf{A mean value theorem for Dirichlet polynomials}}
	In this section, we give some auxiliary lemmas and a proof of \cref{PAFM1}.

	\begin{lemma}	\label{CscrM-INEsRt}
		For any arithmetic function $r$ and any $L \geq 3$, we have
		\begin{align}	\label{INEsRt}
			|R_{L}(\tfrac{1}{2} + \i t)|^{2}
			\leq \sum_{m, n \leq L} \frac{|r(m) r(n)|}{\sqrt{mn}}
			\leq L \sum_{n \leq L}\frac{|r(n)|^{2}}{n}.
		\end{align}
	\end{lemma}

	This is Lemma 3.1 in \cite{IKT2025} except that $r(n)$ in the original statement is replaced by $r(n) n^{-1/2}$.

	\begin{lemma}	\label{lem5MSMA}
		For any arithmetic function $r$ and any $L, T \geq 3$ satisfying $L \leq T / (\log{T})^{2}$, we have
		\begin{align}
			\int_{-\infty}^{\infty}|R_{L}(\tfrac{1}{2} + \i t)|^{2}\Phi_{T}(t)\sd t
			= \hat{\Phi_{T}}(0) \l( 1 + O\l( \frac{1}{T} \r) \r) \sum_{n \leq L} \frac{|r(n)|^{2}}{n}.
		\end{align}
	\end{lemma}

	This corresponds to Lemma 3.2 in \cite{IKT2025} under the same modification for $r(n)$ as above.

	Before proceeding with the proof, we recall that the branch of $\log\zeta(s)$ was fixed in Section 2,
	and define $S(t) = \frac{1}{\pi}\Im\log\zeta(\frac{1}{2} + \i t)$.

	\begin{proof}[Proof of \cref{PAFM1}]
		Since $N_{h}(t), S(t) \ll \log(|t| + 3)$ holds and $\Phi_{T}$ is negligibly small for $t \in \RR \setminus [T, 2T]$, we have
		\begin{align}	\label{SEQM1}
			M_{1}(T, h; R_{L})
			= \int_{T}^{2T} N_{h}(t) |R_{L}(\tfrac{1}{2} + \i t)|^{2} \Phi_{T}(t)\sd t + O\l( \frac{1}{T} \sum_{n \leq L}\frac{|r(n)|^{2}}{n} \r)
		\end{align}
		by \cref{CscrM-INEsRt}.
		By the Riemann--von Mangoldt formula, we have
		\begin{align}	\label{SIRvMF}
			N_{h}(t)
			= \frac{h}{2\pi} \log T + S(t + \tfrac{h}{2}) - S(t - \tfrac{h}{2}) + O\l( h + \frac{1}{T} \r)
		\end{align}
		for $t \in [T, 2T]$, $0 < h \leq 1$.
		Inserting this formula, we see that the integral on the right hand side of \cref{SEQM1} is
		\begin{align}
			&\frac{h}{2\pi}\log T\int_{T}^{2T} |R_{L}(\tfrac{1}{2} + \i t)|^{2} \Phi_{T}(t)\sd t
			+ \int_{T}^{2T} \l\{ S(t + \tfrac{h}{2}) - S(t - \tfrac{h}{2}) \r\} |R_{L}(\tfrac{1}{2} + \i t)|^{2} \Phi_{T}(t)\sd t\\
			&\quad+ O\l( \l(h + \frac{1}{T}\r) \int_{-\infty}^{\infty}|R_{L}(\tfrac{1}{2} + \i t)|^{2} \Phi_{T}(t)\sd t \r).
		\end{align}
		Applying \cref{lem5MSMA} and the rapid decay of $\Phi_{T}$ outside $[T, 2T]$, we obtain
		\begin{align}	\label{p1PAFM1fL}
			&M_{1}(T, h; R_{L}) =\\
			&\frac{h}{2\pi}\log T\int_{-\infty}^{\infty} |R_{L}(\tfrac{1}{2} + \i t)|^{2} \Phi_{T}(t)\sd t
			+ \int_{-\infty}^{\infty} \l\{ S(t + \tfrac{h}{2}) - S(t - \tfrac{h}{2}) \r\} |R_{L}(\tfrac{1}{2} + \i t)|^{2} \Phi_{T}(t)\sd t\\
			&+ O\l( \hat{\Phi_{T}}(0)\l(h + \frac{1}{T}\r)\sum_{n \leq L}\frac{|r(n)|^{2}}{n} \r).
		\end{align}
		Moreover, we use \cref{lem5MSMA} again to find that the first term is
		\begin{align}
			\hat{\Phi_{T}}(0) \frac{h}{2\pi}\log T \sum_{n \leq L}\frac{|r(n)|^{2}}{n}  + O\l( \hat{\Phi_{T}}(0)\frac{1}{T} \sum_{n \leq L}\frac{|r(n)|^{2}}{n} \r).
		\end{align}
		By Proposition 3 in \cite{IKT2025}, we have
		\begin{align}	\label{RISth}
			&\int_{-\infty}^{\infty} \l\{ S(t + \tfrac{h}{2}) - S(t - \tfrac{h}{2}) \r\} |R_{L}(\tfrac{1}{2} + \i t)|^{2} \Phi_{T}(t)\sd t\\
			&= \hat{\Phi_{T}}(0) \frac{1}{\pi}\Im\sum_{km \leq L}\frac{g_{h}(k)}{k} \frac{r(m) \bar{r(km)}}{m} + O\l( \hat{\Phi_{T}}(0)\frac{1}{T}\sum_{n \leq L}\frac{|r(n)|^{2}}{n} \r)
		\end{align}
		when $L \leq T / (\log T)^{2}$.
		Combining these estimates, we obtain \cref{PAFM1}.
	\end{proof}

\section{\textbf{Calculations for $M_{2}(T, h; R_{L})$}}	\label{Sec:CalM2}
	To prove \cref{PAFM2}, we require some asymptotic formulas and inequalities related to $M_{2}(T, h; R_{L})$.
	At the end of this section, we give a proof of \cref{PAFM2}.

	\subsection{A deformation of $M_{2}(T, h; R_{L})$}	\label{subSec:DM2}
		Put
		\begin{align}	\label{def_scrIf}
			\mathcal{I}(T, h; R_{L})
			= \int_{-\infty}^{\infty}\l|\log\zeta(\tfrac{1}{2} + \i(t + \tfrac{h}{2})) - \log\zeta(\tfrac{1}{2} + \i(t - \tfrac{h}{2}))\r|^{2}|R_{L}(\tfrac{1}{2} + \i t)|^{2}\Phi_{T}(t)\sd t,
		\end{align}
		and
		\begin{align}	\label{def_scrJf}
			\mathcal{J}(T, h; R_{L})
			= \int_{-\infty}^{\infty}\l\{\log\zeta(\tfrac{1}{2} + \i(t + \tfrac{h}{2})) - \log\zeta(\tfrac{1}{2} + \i(t - \tfrac{h}{2}))\r\}^{2}|R_{L}(\tfrac{1}{2} + \i t)|^{2}\Phi_{T}(t)\sd t.
		\end{align}
		As the first step toward proving \cref{PAFM2}, we establish the following lemma.

		\begin{lemma}	\label{FDM2DP}
			For any arithmetic function $r$, any large $L, T$ satisfying $L \leq T / (\log T)^{2}$ and any $0 < h \leq 1$, we have
			\begin{align}
				&M_{2}(T, h; R_{L})\\
				&= \l(\frac{h}{\pi}\log T + O(h)\r) M_{1}(T, h; R_{L}) - \hat{\Phi_{T}}(0) \l( \frac{h}{2\pi}\log T \r)^{2} \sum_{n \leq L}\frac{|r(n)|^{2}}{n}\\
				&\quad+\frac{1}{2\pi^{2}}\mathcal{I}(T, h; R_{L}) - \frac{1}{2\pi^{2}} \Re\mathcal{J}(T, h; R_{L})
				+ O\l( \hat{\Phi_{T}}(0)\l( h^{2} \log T + \frac{\log T}{T} \r)\sum_{n \leq L}\frac{|r(n)|^{2}}{n} \r).
			\end{align}
		\end{lemma}

		Thanks to this lemma, our remaining task is to compute the integrals $\mathcal{I}(T, h; R_{L})$ and $\mathcal{J}(T, h; R_{L})$.

		\begin{proof}
			By \cref{SIRvMF} and the estimate $N_{h}(t) \ll \log(|t| + 3)$, we have
			\begin{align}
				N_{h}(t)^{2}
				&= \l(\frac{h}{\pi} \log T + O(h)\r) N_{h}(t) - \l( \frac{h}{2\pi} \log T \r)^{2} + \l( S(t + \tfrac{h}{2}) - S(t - \tfrac{h}{2}) \r)^{2}\\
				&\quad+ O\l( h^{2} \log T + \frac{\log T}{T} \r)
			\end{align}
			for $t \in [T, 2T]$.
			Since $\Phi_{T}$ is negligibly small on $\RR \setminus [T, 2T]$, we find that
			\begin{align}
				&M_{2}(T, h; R_{L})\\
				&= \int_{T}^{2T} N_{h}(t)^{2}|R_{L}(\tfrac{1}{2} + \i t)|^{2} \Phi_{T}(t)\sd t + O\l( \frac{1}{T} \sum_{n \leq L}\frac{|r(n)|^{2}}{n} \r)\\
				&= \int_{T}^{2T} \l\{ \l(\frac{h}{\pi} \log T + O(h)\r) N_{h}(t) - \l(\frac{h}{2\pi} \log T\r)^{2}
				+ \l\{ S(t + \tfrac{h}{2}) - S(t - \tfrac{h}{2}) \r\}^{2} \r\}|R_{L}(\tfrac{1}{2} + \i t)|^{2} \Phi_{T}(t)\sd t\\
				&\quad+ O\l( \int_{T}^{2T}\l(h + \frac{\log T}{T}\r)|R_{L}(\tfrac{1}{2} + \i t)|^{2} \Phi_{T}(t)\sd t + \hat{\Phi_{T}}(0)\frac{1}{T} \sum_{n \leq L}\frac{|r(n)|^{2}}{n} \r)\\
				&= \int_{-\infty}^{\infty} \l\{ \l(\frac{h}{\pi} \log T + O(h)\r) N_{h}(t) - \l(\frac{h}{2\pi} \log T\r)^{2}
				+ \l\{ S(t + \tfrac{h}{2}) - S(t - \tfrac{h}{2}) \r\}^{2} \r\}|R_{L}(\tfrac{1}{2} + \i t)|^{2} \Phi_{T}(t)\sd t\\
				&\quad+ O\l( \hat{\Phi_{T}}(0)\l( h^{2} \log T + \frac{\log T}{T} \r)\sum_{n \leq L}\frac{|r(n)|^{2}}{n} \r).
			\end{align}
			By the definition of $M_{1}$ and \cref{lem5MSMA}, we have
			\begin{align}
				M_{2}(T, h; R_{L})
				&= \l(\frac{h}{\pi}\log T + O(h)\r) M_{1}(T, h; R_{L}) - \hat{\Phi_{T}}(0) \l( \frac{h}{2\pi}\log T \r)^{2} \sum_{n \leq L}\frac{|r(n)|^{2}}{n}\\
				&\quad+\int_{-\infty}^{\infty} \l\{ S(t + \tfrac{h}{2}) - S(t - \tfrac{h}{2}) \r\}^{2}|R_{L}(\tfrac{1}{2} + \i t)|^{2} \Phi_{T}(t)\sd t\\
				&\quad+ O\l( \hat{\Phi_{T}}(0) \l( h^{2} \log T + \frac{\log T}{T} \r) \sum_{n \leq L}\frac{|r(n)|^{2}}{n} \r).
			\end{align}
			Using the basic identity $(\Im z)^{2} = \frac{1}{2}(|z|^{2} - \Re(z^{2}))$ for $z \in \CC$, we complete the proof of \cref{FDM2DP}.
		\end{proof}

	\subsection{Asymptotic formula for $\mathcal{J}(T, h; R_{L})$}
		In this subsection, we prove the following formula.
		\begin{proposition}	\label{AFscrJDP}
			Assume the Riemann Hypothesis.
			For any arithmetic function $r$, any large $T$, and any $0 < h \leq 1$, we have
			\begin{align}
				\mathcal{J}(T, h; R_{L})
				= \hat{\Phi_{T}}(0) \sum_{klm \leq L}\frac{g_{h}(k) g_{h}(l)}{kl} \frac{r(m) \bar{r(klm)}}{m}
				+ O\l( \hat{\Phi_{T}}(0)\frac{1}{T}\sum_{n \leq L}\frac{|r(n)|^{2}}{n} \r).
			\end{align}
		\end{proposition}

		To prove this proposition, we use the following lemma that is an analogue of the formula by Selberg/Tsang \cite{Ts1986} (see \cref{lem_STF} below).

		\begin{lemma}	\label{lem_STF_Sv}
			Let $V$ be an analytic function in the horizontal strip $\set{z \in \CC}{-\frac{3}{2} \leq \Im z \leq 0}$ satisfying
			$
				\ds{\sup_{-\frac{3}{2} \leq y \leq 0}|V(x + \i y)| \ll (|x| \log^{2}(|x| + 3))^{-1}}.
			$
			For any $v \in \RR$, we have
			\begin{align}
				&\int_{-\infty}^{\infty}\l\{\log\zeta(\tfrac{1}{2} + \i(t + \tfrac{v}{2})) - \log\zeta(\tfrac{1}{2} + \i(t - \tfrac{v}{2}))\r\}^{2}V(t)\sd t\\
				&= \sum_{k, l \geq 2} \frac{g_{v}(k) g_{v}(l)}{\sqrt{kl}} \hat{V}\l(\frac{\log(k l)}{2\pi}\r)
				+ \mathscr{Z}(v; V) + \mathscr{P}(v; V),
			\end{align}
			where
			\begin{align}
				\mathscr{Z}(v; V)
				&= \sum_{\b > \frac{1}{2}}\int_{C(\b - 1/2)}\l\{\log\zeta(\tfrac{1}{2} + \i \gamma + w) - \log\zeta(\tfrac{1}{2} + \i(\gamma - v) + w)\r\}^{2}V(\gamma - \tfrac{v}{2} - \i w)\sd w\\
				&\quad+ \sum_{\b > \frac{1}{2}}\int_{C(\b - 1/2)}\l\{\log\zeta(\tfrac{1}{2} + \i (\gamma + v) + w) - \log\zeta(\tfrac{1}{2} + \i \gamma + w)\r\}^{2}V(\gamma + \tfrac{v}{2} - \i w)\sd w,
			\end{align}
			and
			\begin{align}
				\mathscr{P}(v; V)
				&= \int_{C(1/2)}\l\{\log\zeta(\tfrac{1}{2} + w) - \log\zeta(\tfrac{1}{2} - \i v + w)\r\}^{2}V(- \tfrac{v}{2} - \i w)\sd w\\
				&\quad+ \int_{C(1/2)}\l\{\log\zeta(\tfrac{1}{2} + \i v + w) - \log\zeta(\tfrac{1}{2} + w)\r\}^{2}V(\tfrac{v}{2} - \i w)\sd w.
			\end{align}
			Here, $C(\a)$ denotes the contour starting from the origin, proceeding along the real axis to $\a$, encircling $\a$ once counterclockwise, and returning to the origin along the same path.
		\end{lemma}

		\begin{proof}
			Shifting the contour of integration, we find that
			\begin{align}
				&\int_{-\infty}^{\infty}\l\{\log\zeta(\tfrac{1}{2} + \i(t + \tfrac{v}{2})) - \log\zeta(\tfrac{1}{2} + \i(t - \tfrac{v}{2}))\r\}^{2}V(t)\sd t\\
				&= \int_{-\infty}^{\infty}\l\{\log\zeta(2 + \i(t + \tfrac{v}{2})) - \log\zeta(2 + \i(t - \tfrac{v}{2}))\r\}^{2}V(t - \tfrac{3\i}{2} )\sd t
				+ \mathscr{Z}(v; V) + \mathscr{P}(v; V).
			\end{align}
			By the Dirichlet series representation, the first integral on the right hand side is equal to
			\begin{align}
				\sum_{k, l \geq 2}\frac{g_{v}(k) g_{v}(l)}{(kl)^{2}} \int_{-\infty}^{\infty} (k l)^{-\i t}V(t - \tfrac{3\i}{2} )\sd t
			\end{align}
			Moving the path of the integral again, we obtain
			\begin{align}
				\int_{-\infty}^{\infty}\l\{\log\zeta(2 + \i(t + \tfrac{v}{2})) - \log\zeta(2 + \i(t - \tfrac{v}{2}))\r\}^{2}V(t - \tfrac{3\i}{2} )\sd t
				&= \sum_{k, l \geq 2}\frac{g_{v}(k) g_{v}(l)}{\sqrt{kl}} \hat{V}\l(\frac{\log(kl)}{2\pi}\r),
			\end{align}
			which completes the proof of this lemma.
		\end{proof}

		\begin{proof}[Proof of \cref{AFscrJDP}]
			We expand the Dirichlet polynomial $R_{L}$ to obtain
			\begin{align}	\label{p3SPStAbsLRZ}
				&\mathcal{J}(T, h; R_{L})\\
				&= \sum_{m, n \leq L}\frac{r(m) \bar{r(n)}}{\sqrt{mn}}
				\int_{-\infty}^{\infty}\l\{ \log\zeta(\tfrac{1}{2} + \i(t + \tfrac{h}{2})) - \log\zeta(\tfrac{1}{2} + \i(t - \tfrac{h}{2})) \r\}^{2} \l( \frac{m}{n} \r)^{-\i t} \Phi_{T}(t)\sd t.
			\end{align}
			Following the argument of (9.9.5) in \cite{THB1986}, we have $\int_{1/2}^{2}|\log\zeta(\s + \i t)|^{2} \sd\s \ll (\log(|t| + 3))^{2}$.
			We apply \cref{lem_STF_Sv} to the integral in \cref{p3SPStAbsLRZ} and then estimate the contribution coming from $\mathscr{P}$ by using the previous estimate to find that \cref{p3SPStAbsLRZ} is
			\begin{align}	\label{p1SPStAbsLRZ}
				&\sum_{m, n \leq L}\frac{r(m) \bar{r(n)}}{\sqrt{mn}} \sum_{k, l \geq 2} \frac{g_{h}(k) g_{h}(l)}{\sqrt{k l}} \hat{\Phi_{T}}\l(\frac{1}{2\pi}\log\l( \frac{k l m}{n} \r)\r)\\
				&+ O\l( \sum_{m, n \leq L}\frac{|r(m)r(n)|}{\sqrt{mn}}\l\{\l( \frac{m}{n} \r)^{1/2} + \l( \frac{n}{m} \r)^{1/2}\r\} \Phi_{T}(0) \r)
			\end{align}
			under the Riemann Hypothesis.
			Note that the contribution from $\mathscr{Z}$ vanishes under the Riemann Hypothesis.
			This $O$-term is $\ll \hat{\Phi_{T}}(0)\frac{1}{T}\sum_{n \leq L}|r(n)|^{2}n^{-1}$ for $L \leq T / (\log T)^{2}$ by \cref{CscrM-INEsRt}.
			Since it holds that
			\begin{align}	\label{EQhPhi}
				\hat{\Phi_{T}}(\xi) = \sqrt{2\pi} \frac{T}{\log T} e^{-2\pi \i \xi \frac{3}{2}T} \Phi\l( 2\pi\frac{T}{\log T} \xi \r)
				= \hat{\Phi_{T}}(0) e^{-3\pi \i \xi T} \Phi\l( 2\pi\frac{T}{\log T} \xi \r),
			\end{align}
			we have
			\begin{align}	\label{ESThPhi}
				\hat{\Phi_{T}}\l( \frac{1}{2\pi}\log\l( \frac{M}{N} \r) \r)
				\ll \Phi\l( \frac{T}{2 L \log T} \r)
			\end{align}
			for any large $L$ and any $M, N \in \ZZ_{\geq 1} \cap [1, L]$ with $M \neq N$.
			This estimate ensures that the non-diagonal terms are negligible when $L \leq T / (\log T)^{2}$.
			Hence, the main term of \cref{p1SPStAbsLRZ} is
			\begin{align}
				\hat{\Phi_{T}}(0) \sum_{klm \leq L}\frac{r(m) \bar{r(klm)}}{m} \frac{g_{h}(k) g_{h}(l)}{kl}
				+ O\l( \hat{\Phi_{T}}(0)\frac{1}{T}\sum_{n \leq L}\frac{|r(n)|^{2}}{n} \r).
			\end{align}
			This completes the proof of \cref{AFscrJDP}.
		\end{proof}

	\subsection{Lower bound of $\mathcal{I}(T, h; R_{L})$}
		In this subsection, we use the approximator $A_{X}$ defined in \cref{def_AX} to establish the following proposition.

		\begin{proposition}	\label{RMSSt}
			Assume the Riemann Hypothesis.
			For any arithmetic functions $a, r$, any large $L, T$ satisfying $L \leq T / (\log T)^{2}$, and any $0 < h \leq 1$, we have
			\begin{align}
				\mathcal{I}(T, h; R_{L})
				&\geq \hat{\Phi_{T}}(0)\Re\sum_{\substack{km, ln \leq L\\ km = ln}} \frac{(2g_{h}(k) - a(k)) \bar{a(l)}}{\sqrt{kl}} \frac{r(m) \bar{r(n)}}{\sqrt{mn}}\\
				&\quad- O\l( \frac{1}{T}\sum_{k \leq L}\frac{|a(k)|^{2}}{k}\sum_{n \leq L}\frac{|r(n)|^{2}}{n} \r).
			\end{align}
		\end{proposition}

		By \cref{KINEDI}, it suffices to analyze $\mathcal{I}_{1}(T; R_{L}, A)$ and $\mathcal{I}_{2}(T, h; R_{L}, A)$ defined in \cref{def_calI1} and \cref{def_calI2}, respectively.
		Here, we show the asymptotic formulas for $\mathcal{I}_{1}(T; R_{L}, A)$ and $\mathcal{I}_{2}(T, h; R_{L}, A)$.

		\begin{lemma}	\label{EST_TDPR}
			For any arithmetic functions $a, r$, any large $L, T$ satisfying $L \leq T / (\log T)^{2}$, and any $0 < h \leq 1$, we have
			\begin{align}
				\mathcal{I}_{1}(T; R_{L}, A)
				&= \hat{\Phi_{T}}(0)\sum_{\substack{km, ln \leq L\\ km = ln}} \frac{a(k) \bar{a(l)}}{\sqrt{kl}} \frac{r(m) \bar{r(n)}}{\sqrt{mn}}
				+ O\l( \frac{1}{T}\sum_{k \leq L}\frac{|a(k)|^{2}}{k}\sum_{n \leq L}\frac{|r(n)|^{2}}{n} \r).
			\end{align}
		\end{lemma}

		\begin{lemma}	\label{CSthCDP}
			Assume the Riemann Hypothesis.
			For any arithmetic functions $a, r$, any large $L, T$ satisfying $L \leq T / (\log T)^{2}$, and any $0 < h \leq 1$, we have
			\begin{align}
				\mathcal{I}_{2}(T, h; R_{L}, A)
				&= \hat{\Phi_{T}}(0)\sum_{\substack{km, ln \leq L\\ km = ln}} \frac{g_{h}(k) \bar{a(l)}}{\sqrt{kl}} \frac{r(m) \bar{r(n)}}{\sqrt{mn}}
				+ O\l( \frac{1}{T}\sum_{k \leq L}\frac{|a(k)|^{2}}{k}\sum_{n \leq L}\frac{|r(n)|^{2}}{n} \r).
			\end{align}
		\end{lemma}

		Using these lemmas and inequality \cref{KINEDI}, we obtain \cref{RMSSt}, noting that the main term of \cref{EST_TDPR} is real.

	\subsection{Proof of \cref{EST_TDPR}}
		Expanding the Dirichlet polynomials $R_{L}$ and $A_{X}$, we write
		\begin{align}
			\mathcal{I}_{1}(T; R_{L}, A)
			&= \sum_{km, ln \leq L} \frac{a(k) \bar{a(l)}}{\sqrt{kl}} \frac{r(m) \bar{r(n)}}{\sqrt{mn}} \int_{-\infty}^{\infty} \l( \frac{k m}{l n} \r)^{-\i t} \Phi_{T}(t)\sd t\\
			&= \sum_{km, ln \leq L} \frac{a(k) \bar{a(l)}}{\sqrt{kl}} \frac{r(m) \bar{r(n)}}{\sqrt{mn}} \hat{\Phi_{T}}\l( \frac{1}{2\pi}\log\l( \frac{km}{ln} \r) \r).
		\end{align}
		Applying \cref{EQhPhi} and \cref{ESThPhi}, we obtain
		\begin{align}
			\mathcal{I}_{1}(T; R_{L}, A)
			&= \hat{\Phi_{T}}(0)\sum_{\substack{km, ln \leq L\\ km = ln}} \frac{a(k) \bar{a(l)}}{\sqrt{kl}} \frac{r(m) \bar{r(n)}}{\sqrt{mn}}
			+ O\l( \frac{1}{T}\sum_{k \leq L}\frac{|a(k)|^{2}}{k}\sum_{n \leq L}\frac{|r(n)|^{2}}{n} \r).
		\end{align}
		This completes the proof of \cref{EST_TDPR}.
		\qed

	\subsection{Proof of \cref{CSthCDP}}
		To prove \cref{CSthCDP}, we require the following lemma.

		\begin{lemma}	\label{lem_STF}
			Let $V$ be an analytic function in the horizontal strip $\set{z \in \CC}{-\frac{3}{2} \leq \Im z \leq 0}$ satisfying
			$
				\ds{\sup_{-\frac{3}{2} \leq y \leq 0}|V(x + \i y)| \ll (|x| \log^{2}(|x| + 3))^{-1}}.
			$
			For any $v \in \RR$, we have
			\begin{align}
				&\int_{-\infty}^{\infty}\log\zeta(\tfrac{1}{2} + \i(t + v))V(t)\sd t\\
				&= \sum_{k = 2}^{\infty}\frac{\Lam(k)}{k^{\frac{1}{2} + \i v}\log{k}}\hat{V}\l( \frac{\log{k}}{2\pi} \r)
				+ 2\pi \sum_{\b > \frac{1}{2}}\int_{0}^{\b - \frac{1}{2}}V(\gamma - v - \i\s)\sd\s
				- 2\pi \int_{0}^{\frac{1}{2}}V(- v - \i\s)\sd\s.
			\end{align}
			Here, $\hat{V}$ is the Fourier transform of $V$ defined by $\hat{V}(z) = \int_{-\infty}^{\infty}V(x)e^{-2\pi \i x z}\sd x$.
		\end{lemma}

		\begin{proof}
			This follows from (2.14) in \cite{Ts1986} and the subsequent argument.
		\end{proof}

		\begin{proof}[Proof of \cref{CSthCDP}]
			We write
			\begin{align}
				&\mathcal{I}_{2}(T, h; R_{L}, A)\\
				&= \sum_{m, ln \leq L} \frac{\bar{a(l)} r(m) \bar{r(n)}}{\sqrt{lmn}}
				\int_{-\infty}^{\infty} \l\{ \log\zeta(\tfrac{1}{2} + \i(t + \tfrac{h}{2})) - \log\zeta(\tfrac{1}{2} + \i(t - \tfrac{h}{2})) \r\}\l(\frac{m}{ln} \r)^{-\i t} \Phi_{T}(t) \sd t.
			\end{align}
			It holds by \cref{lem_STF} that for any $x > 0$
			\begin{align}
				&\int_{-\infty}^{\infty} \l\{\log\zeta(\tfrac{1}{2} + \i(t + \tfrac{h}{2})) - \log\zeta(\tfrac{1}{2} + \i(t - \tfrac{h}{2}))\r\} x^{-\i t} \Phi_{T}(t) \sd t\\
				\label{p1CSthCDP}
				&= \sum_{k = 2}^{\infty}\frac{g_{h}(k)}{\sqrt{k}} \hat{\Phi_{T}}\l( \frac{1}{2\pi} \log(k x) \r)
				+ O\l( \l(1 + x^{-1/2}\r)\Phi_{T}(0) \r)
			\end{align}
			under the Riemann Hypothesis.
			Therefore, we have
			\begin{align}
				\mathcal{I}_{2}(T, h; R_{L}, A)
				&= \sum_{m, ln \leq L} \frac{\bar{a(l)} r(m) \bar{r(n)}}{\sqrt{lmn}} \sum_{k = 2}^{\infty}\frac{g_{h}(k)}{\sqrt{k}} \hat{\Phi_{T}}\l( \frac{1}{2\pi} \log\l(\frac{km}{ln}\r) \r)\\
				&\quad + O\l( \Phi_{T}(0) \sum_{m, ln \leq L}\frac{|a(l) r(m) r(n)|}{\sqrt{lmn}}\l\{ 1 + (ln / m)^{1/2}\r\} \r).
			\end{align}
			Moreover, applying \cref{EQhPhi} and \cref{ESThPhi} to the main term and \cref{lem5MSMA} to the error term, we find that
			\begin{align}
				\mathcal{I}_{2}(T, h; R_{L}, A)
				&= \hat{\Phi_{T}}(0)\sum_{\substack{km, ln \leq L\\ km = ln}} \frac{g_{h}(k) \bar{a(l)}}{\sqrt{kl}} \frac{r(m) \bar{r(n)}}{\sqrt{mn}}
				+ O\l( \frac{1}{T}\sum_{k \leq L}\frac{|a(k)|^{2}}{k} \sum_{n \leq L}\frac{|r(n)|^{2}}{n}\r),
			\end{align}
			which completes the proof of \cref{CSthCDP}.
		\end{proof}

	\subsection{Proof of \cref{PAFM2}}
		Combining \cref{FDM2DP}, \cref{AFscrJDP}, and \cref{RMSSt}, we complete the proof of \cref{PAFM2}.
		\qed

\section{\textbf{Completion of the proof of \cref{Nh2-Nh}}}	\label{Sec:CR}
	In this section, we give the proof of \cref{SGZuRH}.
	To do so, we establish an inequality for the maximum of $N_{h}(t)^{2} - N_{h}(t)$ and choose the coefficients of the resonator and the approximator suitably.

	\subsection{Inequality of extreme values of $N_{h}(t)^{2} - N_{h}(t)$}
		As the first step proving \cref{SGZuRH}, we establish the following proposition.
		\begin{proposition}	\label{CNhNh2uRH}
			Assume the Riemann Hypothesis.
			Let $a, r$ be arithmetic functions not identically zero.
			For any large $L, T$ satisfying $L \leq T / (\log T)^{2}$ and any $0 < h \leq 1$, we have
			\begin{align}	\label{CNhNh2uRH1}
				&\sup_{t \in [T, 2T]} \l( N_{h}(t)^{2} - N_{h}(t) \r)\\
				&\geq \l( 1 - O\l( \frac{1}{T} \r) \r) \l\{\mathcal{N}(T, h, L; a, r) \Bigg/ \sum_{n \leq L}\frac{|r(n)|^{2}}{n}\r\}
				- \frac{h}{2\pi}\log T \l( 1 - \frac{h}{2\pi}\log T \r)\\
				&\quad- O\l( h^{2} \log T + \frac{\log T}{T} + \frac{1}{T}\sum_{k \leq L}\frac{|a(k)|^{2}}{k} \r),
			\end{align}
			where
			\begin{align}
				\mathcal{N}(T, h, L; a, r)
				&= -\l( 1 - \frac{h}{\pi}\log T + O(h) \r)\frac{1}{\pi}\Im\sum_{km \leq L}\frac{g_{h}(k)}{k}\frac{r(m) \bar{r(km)}}{m}\\
				&\quad + \frac{1}{2\pi^{2}}\Re\sum_{\substack{km, ln \leq L\\ km = ln}} \frac{(2g_{h}(k) - a(k)) \bar{a(l)}}{\sqrt{kl}} \frac{r(m) \bar{r(n)}}{\sqrt{mn}}\\
				&\quad - \frac{1}{2\pi^{2}}\Re\sum_{klm \leq L}\frac{g_{h}(k) g_{h}(l)}{kl} \frac{r(m) \bar{r(klm)}}{m}.
			\end{align}
			Here, the implicit constants are absolute.
		\end{proposition}

		\begin{proof}
			Following the same argument as in \cref{SEQM1}, we have
			\begin{align}
				M_{2}(T, h; R_{L})
				= \int_{T}^{2T} N_{h}(t)^{2} |R_{L}(\tfrac{1}{2} + \i t)|^{2} \Phi_{T}(t)\sd t + O\l( \frac{1}{T} \sum_{n \leq L}\frac{|r(n)|^{2}}{n} \r).
			\end{align}
			On the other hand, we find by \cref{lem5MSMA} that
			\begin{align}
				&\int_{T}^{2T} (N_{h}(t)^{2} - N_{h}(t)) |R_{L}(\tfrac{1}{2} + \i t)|^{2} \Phi_{T}(t)\sd t\\
				&\leq \sup_{t \in [T, 2T]}(N_{h}(t)^{2} - N_{h}(t)) \int_{-\infty}^{\infty} |R(\tfrac{1}{2} + \i t)|^{2} \Phi_{T}(t)\sd t\\
				&\leq \sup_{t \in [T, 2T]}(N_{h}(t)^{2} - N_{h}(t)) \hat{\Phi_{T}}(0) \l( 1 + O\l( \frac{1}{T} \r) \r) \sum_{n \leq L} \frac{|r(n)|^{2}}{n}.
			\end{align}
			Combining these with \cref{SEQM1}, we have
			\begin{align}
				&\sup_{t \in [T, 2T]}(N_{h}(t)^{2} - N_{h}(t))\\
				&\geq \l( 1 - O\l( \frac{1}{T} \r) \r) \l\{ M_{2}(T, h; R_{L}) - M_{1}(T, h; R_{L}) \r\} \bigg/ \l\{\hat{\Phi_{T}}(0)\sum_{n \leq L}\frac{|r(n)|^{2}}{n}\r\}
				- O\l( \frac{1}{T} \r).
			\end{align}
			Applying \cref{PAFM1,PAFM2} to the right hand side, we obtain \cref{CNhNh2uRH}.
		\end{proof}

	\subsection{Choices of the resonator and the approximator}
		We set the coefficients of the approximator $A_{X}(s)$ to be $a(k) = g_{h}(k)$.
		Let $d_{\ell}(n)$ be the generalized divisor function, which is multiplicative and defined by
		\begin{align}
			d_{\ell}(p^{a})
			= \frac{\Gamma(a + \ell)}{\Gamma(\ell) \Gamma(a + 1)},
		\end{align}
		and let $\lam(n)$ be the Liouville function.
		We then choose the coefficients of the resonator $r(n)$ to be
		\begin{align}
			r(n) = \l\{
			\begin{array}{cl}
				d_{\ell}(n) \lam(n) f(\log n / \log L) & \text{if \, $\frac{h}{\pi} \log T \geq 1$,} \\
				d_{\ell}(n) f(\log n / \log L)         & \text{if \, $\frac{h}{\pi} \log T < 1$.}
			\end{array}
			\r.
		\end{align}
		Here, $f$ is a continuous real-valued function of bounded variation on $[0, 1]$ satisfying $I_{f}(\ell) \ceq \int_{0}^{1}f(u)^{2} u^{\ell^{2} - 1}\sd u \neq 0$.
		This choice is based on the work \cite{BMN2010}.
		For the above $a$, $r$, $f$, we have the following lemma.

		\begin{lemma}	\label{AFH}
			Let $a$, $r$, $f$ be as above.
			Let $\ell \geq 1$.
			For any large $L, T$ and any $0 < h \leq 1$, we have
			\begin{align}
				&\l\{\mathcal{N}(T, L, h; a, r) \Bigg/ \sum_{n \leq L}\frac{|r(n)|^{2}}{n}\r\}\\
				&= \l| 1 - \frac{h}{\pi}\log T \r| \frac{2}{\pi} \ell I_{f}(\ell)^{-1} \int_{0}^{1}\frac{\sin(\tfrac{u}{2} h \log L)}{u} \int_{0}^{1 - u} f(v) f(u + v) v^{\ell^{2} - 1}\sd v \sd u\\
				&\quad+ \frac{2}{\pi^{2}} \ell^{2} I_{f}(\ell)^{-1} \int_{0}^{1}\frac{\sin(\frac{u_{1}}{2} h \log L)}{u_{1}} \int_{0}^{1 - u_{1}}\frac{\sin(\frac{u_{2}}{2} h \log L)}{u_{2}}\\
				&\qquadn{2} \times \int_{0}^{1 - (u_{1} + u_{2})} \l\{f(v) f(u_{1} + u_{2} + v) + f(u_{1} + v) f(u_{2} + v)\r\} v^{\ell^{2} - 1}\sd v \sd u_{2} \sd u_{1}\\
				&\quad+\frac{2}{\pi^{2}} I_{f}(\ell)^{-1} \int_{0}^{1} \frac{\sin^{2}(\tfrac{u}{2} h \log L)}{u}\int_{0}^{1 - u} f(v)^{2} v^{\ell^{2} - 1}\sd v \sd u
				+ O_{\ell, f}\l( h + h^{2} \log L \r).
			\end{align}
		\end{lemma}

		We omit the proof of this lemma, as it can be proved by a routine calculation similar to that of Lemma 2.1 in \cite{BMN2010} using partial summation, the prime number theorem, and the formula
		\begin{align}	\label{SMdl}
			\sum_{n \leq X}\frac{d_{\ell}(n)^{2}}{n}
			= C_{\ell} (\log X)^{\ell^{2}} + O_{\ell}\l( (\log X)^{\ell^{2} - 1} \r)
		\end{align}
		for $X \geq 3$, $\ell \geq 1$, where $C_{\ell}$ is a positive constant depending only on $\ell$.
		Asymptotic formula \cref{SMdl} is well known and can be proved by following the argument in Section 1.6 of \cite{IK}.

	\subsection{Proof of \cref{Nh2-Nh}}
		We apply \cref{AFH} in the case $h = 2\pi \phi / \log T$, $L = T / (\log T)^{2}$ to \cref{CNhNh2uRH}.
		We should remark that the $O$-term in \cref{Nh2-Nh} comes from the error term in \cref{AFH} and the choice $L = T / (\log T)^{2}$ since we have
		\begin{align}
			&\int_{0}^{1}\frac{\sin(\tfrac{u}{2} h \log L)}{u} \int_{0}^{1 - u} f(v) f(u + v) v^{\ell^{2} - 1}\sd v \sd u\\
			&= \int_{0}^{1}\frac{\sin(\pi \phi u)}{u} \int_{0}^{1 - u} f(v) f(u + v) v^{\ell^{2} - 1}\sd v \sd u
			+ O_{\ell, f}\l( \frac{\phi \log\log T}{\log T} \r).
		\end{align}
		Other terms can be calculated in the same way.
		Thus, we complete the proof of \cref{Nh2-Nh}.
		\qed

		\begin{acknowledgment*}
			The author is supported by JSPS KAKENHI Grant Number 24K16907.
		\end{acknowledgment*}

		\end{document}